\documentclass[12pt]{article}
\usepackage{geometry} 
\usepackage[ utf8 ]{inputenc}
\usepackage[T1]{fontenc}
\usepackage[english,french]{babel}
\usepackage{amsmath}
\usepackage{amsfonts}
\usepackage{amssymb}
\usepackage{amsthm}
\usepackage{textcomp}
\usepackage{eucal}
\usepackage{array}
\theoremstyle{plain}
\newtheorem{theorem}{Théorème}

\newtheorem{lemma}{Lemme}
\newtheorem{proposition}{Proposition}
\newtheorem{notation}{Notation}
\theoremstyle{definition}

\theoremstyle{remark}
\newtheorem{remark}{Remarque}
\date{}

\title{ Représentations de réflexion de groupes de Coxeter\\Troisième partie: les groupes diédraux affines}
\author{François ZARA
}
\begin{document}
\maketitle
\begin{abstract}
Dans cette troisième partie, on fait les hypothèses suivantes: la représentation $R=R(\alpha,\beta,\gamma ;l)$ de $W(p,q,r)$ est réductible et il existe une forme bilinéaire non nulle $G$-invariante où $G=Im R$. Dans ces conditions, la structure de $G$ est connue: $G'=G/N(G)$ est isomorphe à un groupe diédral fini et $N(G)$ est donné explicitement ainsi que l'opération de $G$ sur $N(G)$. On commence par donner des conditions sur les $p,q,r$ ainsi que sur $\alpha,\beta,\gamma$ et on les démontre dans l'appendice. On étudiera le cas général dans la partie suivante.
\end{abstract}
\begin{otherlanguage}{english}
\begin{abstract}
In this third part, we make the following hypothesis: representation $R=R(\alpha,\beta,\gamma ;l)$ of $W(p,q,r)$ is reducible and there exist a $G$-invariant non-nulle bilinear form where $G=Im R$. With those conditions, we know the structure of $G$: $G'=G/N(G)$ is isomorphic to a finite dihedral group and $N(G)$ is given explicitly as well as the action of $G$ on $N(G)$. We begin by giving conditions on $p,q,r$ as well on $\alpha,\beta,\gamma$ and we proove them in the appendix. The general case will be studied in the next part.
\end{abstract}
\end{otherlanguage}
\let\thefootnote\relax\footnote{Mots clés et phrases: groupes de Coxeter, groupes de réflexion, groupes diédraux.}
\let\thefootnote\relax
\section{Introduction}
Dans cette partie, on suppose que la représentation $R(\alpha,\beta,\gamma;l)$ de $W(p,q,r)$ est réductible et qu'il existe une forme bilinéaire $G$-invariante non nulle. Alors dans ces conditions, on montre que $G'=G/N(G)$ est isomorphe à un groupe diédral fini et $N(G)$ est donné explicitement  ainsi que l'opération de $G$ sur $N(G)$.  The general case will
\subsection{Contraintes sur les paramêtres}
Les conditions énoncées ci-dessus sont équivalentes aux deux conditions suivantes:
\[
\Delta=8-2\alpha-2\beta-2\gamma-\alpha l-\beta m=0\quad \text{et} \quad\alpha l=\beta m
\]
Nous allons voir que cela implique de fortes contraintes sur $p$, $q$, $r$ d'une part et sur $\alpha$, $\beta$, $\gamma$ d'autre part. Pour cela, nous montrons d'abord le résultat suivant:
\begin{proposition}
Soient $a$, $b$ et $c$ trois nombres réels. On pose: $\alpha:=4\cos^{2}a\pi$, $\beta:=4\cos^{2}b\pi$ et $\gamma:=4\cos^{2}c\pi$. Alors les deux conditions suivantes sont équivalentes:
\begin{enumerate}
  \item (1) $\exists \epsilon,\epsilon' \in \{-1,+1\}$ tels que $c\equiv \epsilon a+\epsilon' b \mod{\mathbb{Z}}$;
  \item (2) $\alpha\beta\gamma=(4-\alpha-\beta-\gamma)^{2}$.
\end{enumerate}
\end{proposition}
\begin{proof}
1) Montrons que (1) implique (2).\\
Il existe $\epsilon_{1}\in \{-1,+1\}$ tel que $\epsilon_{1}\cos c\pi=\cos a\pi \cos b\pi-\epsilon\epsilon'\sin a\pi \sin b\pi$.\\
Nous obtenons:
\begin{align*}
\cos^{2}c\pi=\cos^{2}a\pi \cos^{2}b\pi + \sin^{2}a\pi \sin^{2}b\pi -2\epsilon\epsilon'\cos a\pi \sin a\pi \cos b\pi sin b\pi\\
=1-\cos^{2}a\pi - \cos^{2}b\pi +2\cos^{2}a\pi cos^{2}b\pi -2\epsilon\epsilon'\cos a\pi \sin a\pi \cos b\pi sin b\pi
\end{align*}
d'où, en multipliant cette égalité par $8$:
\[
2\gamma=8-2\alpha-2\beta+\alpha\beta-4\epsilon\epsilon' \sin2a\pi \sin 2b\pi.
\]
Nous en déduisons
\begin{equation*}
\begin{split}
16\sin^{2}2a\pi \sin^{2}2b\pi\\
&=(8-2\alpha-2\beta-2\gamma+\alpha\beta)^{2}\\
&=16(1-2\cos^{2}2a\pi)(1-2cos^{2}2b\pi)\\
&=16(1-\cos 2a\pi)(1+\cos 2a\pi)(1-\cos 2b\pi)(1+\cos 2b/pi)\\
&=16(2-2\cos^{2}a\pi)2\cos^{2}a\pi (2-2\cos^{2}b\pi)2\cos^{2}b\pi\\
&=(4-\alpha)\alpha(4-\beta)\beta\\
&=16\alpha\beta-4\alpha^{2}\beta-4\alpha\beta^{2}+\alpha^{2}\beta^{2}\\
&=(8-2\alpha-2\beta-2\gamma)^{2}+\alpha^{2}\beta^{2}+2\alpha\beta(8-2\alpha-2\beta-2\gamma)
\end{split}
\end{equation*}
d'où, après simplifications
\[
4\alpha\beta\gamma=(8-2\alpha-2\beta-2\gamma)^{2}=4(4-\alpha-\beta-\gamma)^{2}
\]
c'est la relation (2)\\
2) Montrons que (2) implique (1).\\
Posons, pour simplifier l'écriture: $u:=\cos a\pi$, $v:=\cos b\pi$ et $w:= \cos c\pi$ de telle sorte que $\alpha=4u^{2}$, $\beta=4v^{2}$ et $\gamma=4w^{2}$. La relation (2) devient 
\[
64u^{2}v^{2}w^{2}=16(1-u^{2}-v^{2}-w^{2})^{2}
\]
d'où
\[
(1-u^{2}-v^{2}-w^{2})^{2}-4u^{2}v^{2}w^{2}=0=(1-u^{2}-v^{2}-w^{2}+2uvw)(1-u^{2}-v^{2}-w^{2}-2uvw).
\]
Nous brisons maintenant la symétrie entre $u$, $v$ et $w$ pour obtenir:
\[
((w-uv)^{2}-(1-u^{2})(1-v^{2}))((w+uv)^{2}-(1-u^{2})(1-v^{2}))=0.
\]
D'après les valeurs de $u$, $v$ et $w$, nous voyons que $uv=\cos a\pi \cos b\pi$ et $(1-u^{2})(1-v^{2})=\sin^{2}a\pi \sin^{2}b\pi$, donc:
\[
0=((w-\cos a\pi \cos b\pi)^{2}-(\sin a\pi \sin b\pi)^{2})((w+\cos a\pi \cos b\pi)^{2}-(\sin a\pi \sin b\pi)^{2})
\]et $0=d_{1}d_{2}d_{3}d_{4}$ où:
\[
\begin{aligned}
d_{1} &= w-(\cos a\pi \cos b\pi+\sin a\pi \sin b\pi) &= w-\cos(a-b)\pi\\
d_{2} &= w-(\cos a\pi \cos b\pi-\sin a\pi \sin b\pi) &= w-\cos(a+b)\pi\\
d_{3} &=w+(\cos a\pi \cos b\pi+\sin a\pi \sin b\pi) &= w+\cos(a-b)\pi\\
d_{4} &=w+(\cos a\pi \cos b\pi-\sin a\pi \sin b\pi) &= w+\cos(a+b)\pi\\
\end{aligned}
\]
Comme $\cos^{2}x-\cos^{2}y=\sin (x+y)\sin(y-x)$, nous obtenons:
\[
\begin{aligned}
d_{1}d_{3} &= w^{2}-\cos^{2}(a-b)\pi &= \sin(c+a-b)\pi \sin(a-b-c)\pi\\
d_{2}d_{4} &= w^{2}-\cos^{2}(a+b)\pi &= \sin(c+a+b)\pi \sin(a+b-c)\pi
\end{aligned}
\]
d'où finalement:
\[
0=\sin(c+a-b)\pi \sin(a-b-c)\pi\sin(c+a+b)\pi \sin(a+b-c)\pi,
\]
comme $\sin \pi x =0$ si et seulement si $x$ est entier, nous voyons qu'il existe $\epsilon$ et $\epsilon'$ dans $\{-1,+1\}$ tels que $c\equiv \epsilon a+\epsilon' b \mod{\mathbb{Z}}$ c'est la condition (1).
\end{proof}

Pour expliciter les relations entre $p$, $q$, $r$ d'une part et $\alpha$, $\beta$, $\gamma$ d'autre part, nous utilisons le résultat suivant démontré dans l'appendice :
Les deux conditions suivantes $(C)$ et $(D)$ sur le triple d'entiers non nuls $(a_{1},a_{2},a_{3})$ sont équivalentes:
\begin{align*}
(C)
\begin{cases}
(C_{1}) & n=ppcm(a_{1},a_{2},a_{3})=ppcm(a_{i},a_{j}) (1\leqslant i\neq j \leqslant 3);\\
(C_{2}) & \parbox{11 cm}{%
$\exists i,j \in \mathbb{N}$ tels que $(1\leqslant i\neq j \leqslant 3)$ et $v_{2}(a_{i})=v_{2}(a_{j})=v_{2}(n)$;
si $|\{i,j,k\}|=3$, $v_{2}(a_{k})<v_{2}(n)$.}
\end{cases}\\
(D)
\begin{cases}
\exists c_{i}  \in \mathbb{Z} (1\leqslant i \leqslant 3) & \text{tels que}\\
(D_{1}) & \sum_{i=1}^{i=3}\frac{c_{i}}{a_{i}}\in \mathbb{Z};\\
(D_{2}) & pgcd(c_{i},a_{i})=1 (1\leqslant i \leqslant 3).
\end{cases}
\end{align*}
(Si $p$ est un nombre premier et $a$ un entier, $v_{p}(a)$ désigne l'exposant de la plus grande puissance de $p$ qui divise $a$, à ne pas confondre avec le polynôme $v_{p}(X)$!).\\
La condition ($C_{1}$) implique que si $p$ est un nombre premier qui divise $n$, alors il existe des entiers distincts $i$ et $j$ tels que $v_{p}(n)=v_{p}(a_{i})=v_{p}(a_{j})$ et si $|\{i,j,k\}|=3$, alors $v_{p}(a_{k})\leqslant v_{p}(n)$.
\begin{proposition}
Si $\Delta=0$ et $\alpha l=\beta m$, alors le triple d'entiers satisfait à la condition (C).\\
Réciproquement, si le triple d'entiers $(p,q,r)$ satisfait à la condition (C), alors étant donné $\alpha$ une racine de $v_{p}(X)$, on peut trouver $\beta$ racine de $v_{q}(X)$ et $\gamma$ racine de $v_{r}(X)$ tels que l'on ait $\Delta=0$ et $\alpha l=\beta m$.
\end{proposition}
\begin{proof}
Si $\Delta=0$ alors $\alpha l$ et $\beta m$ sont racines du polynôme 
\[
Q(X)=X^{2}-2(4-\alpha-\beta-\gamma)X+\alpha\beta\gamma
\]
et l'on aura $\alpha l=\beta m$ si et seulement si son discriminant $(4-\alpha-\beta-\gamma)^{2}-\alpha\beta\gamma$ est nul.\\
Comme $\alpha=4\cos^{2}\frac{k_{1}\pi}{p}$, $\beta=4\cos^{2}\frac{k_{2}\pi}{q}$, $\gamma=4\cos^{2}\frac{k_{3}\pi}{r}$ avec 
\[
pgcd(k_{1},p)=pgcd(k_{2},q)=pgcd(k_{3},r)=1
\]
d'après la proposition 1, il existe $\epsilon$ et $\epsilon'$ dans $\{-1,+1\}$ tels que 
\[
\frac{k_{3}}{r}-(\epsilon\frac{k_{1}}{p}+\epsilon'\frac{k_{2}}{q})\in \mathbb{Z}.
\]
C'est la condition (D), donc $(p,q,r)$ satisfait à la condition (C).\\
Réciproquement, d'après ce qui précède nous avons $(4-\alpha-\beta-\gamma)^{2}=\alpha\beta\gamma$ et $l$ est tel que $\alpha l=4-\alpha-\beta-\gamma$ donc $\alpha l(4-\alpha-\beta-\gamma)=\alpha\beta lm$ d'où $\beta m=(4-\alpha-\beta-\gamma)$. Dans ces conditions $\Delta=0$ et $\alpha l= \beta m$. la précision sur $\alpha$, $\beta$ et $\gamma$ vient des résultats de l'appendice.
\end{proof}
\begin{remark}
Si $r=2$ et si $\Delta=0$, alors $l=m=0$, donc $\alpha l= \beta m$. La condition (C) exprime que l'on a l'une des deux possibilités suivantes, où l'on a supposé que $p\leqslant q$:
\begin{itemize}
  \item $p$ impair et $q=2p$;
  \item $p$ pair, alors $q=p$ et $4|p$.
 \end{itemize}
\end{remark}
\subsection{Le groupe $G'$}
Pour trouver la structure de $G'$ et de $G$ lorsque $\Delta=0$ et $\alpha l=\beta m$, nous aurons besoin de certains résultats préliminaires.\\
\begin{proposition}
Soient $H$ un groupe commutatif fini non d'exposant $2$ et $\sigma$ un automorphisme de $H$ tel que $\forall h \in H,\sigma(h)=h^{-1}$. On pose $L:=H\rtimes<\sigma>$. Alors $L$ possède une représentation fidèle irréductible $T$ sur un corps algébriquement clos de caractéristique $0$ si et seulement si $H$ est un groupe cyclique (et donc $L$ est un groupe diêdral). Dans ce cas $\deg T=2$.
\end{proposition}
\begin{proof}
Nous allons utiliser le théorème suivant de Gasch\"utz (\cite{G}):"Un groupe fini possède une représentation irréductible et fidèle si et seulement si son socle est engendré par une classe de conjugaison".\\
Soit $T$ une représentation irréductible de $L$. Comme $H$ est un sous-groupe normal commutatif de $L$, nous savons que $\deg T\leqslant 2$. Si $\deg T=1$, alors $D(L)\subset \ker T$ et comme $D(L) \neq {1}$ car $L$ n'est pas commutatif, nous voyons que $T$ n'est pas fidèle dans ce cas. Ainsi $\deg T=2$.

Soit $S$ le socle de $L$. Il est clair que $S \subset H$. Pour chaque $s$ dans $S$, nous avons $\sigma s \sigma^{-1}=s \,\text{ou} \,\sigma s \sigma^{-1}=s^{-1}$, donc la classe de conjugaison de $s$ est $\{s\}$ ou $\{s,s^{-1}\}$. Nous voyons ainsi, grâce au théorème de Gasch\"utz que $L$ possède une représentation fidèle et irréductible si et seulement si $S$ est un groupe cyclique. D'après la structure des groupes commutatifs finis, $S$ est un groupe cyclique si et seulement si $H$ est groupe cyclique. Dans ce cas $L$ est un groupe diédral (d'ordre $\geqslant 6$).
\end{proof}
\begin{notation}
On suppose que le triple d'entiers non nuls $(p,q,r)$ satisfait à la condition $(C_{1})$. On pose $n:=ppcm(p,q,r)$, $d:= pgcd(p,q,r)$, $n=pp_{1}=qq_{1}=rr_{1}$ de telle sorte que $pgcd(p_{1},q_{1})=pgcd(q_{1},r_{1})=pgcd(r_{1},p_{1})=1$, $n=p_{1}q_{1}r_{1}d$, $p=q_{1}r_{1}d$, $q=r_{1}p_{1}d$ et $r=p_{1}q_{1}d$.
\end{notation}
\begin{proposition}
Soit $G"$ le quotient de $W(p,q,r)$ (avec comme système générateur canonique $(t_{1},t_{2},t_{3})$) obtenu en ajoutant la relation $(t_{1}t_{2}t_{3})^{2}=1$.\\ Soit $H":=<t_{1}t_{2},t_{1}t_{3}>$. Alors:\\
1) $H"$ est un sous-groupe commutatif normal de $G"$.\\
2)Le triple d'entiers $(p,q,r)$ satisfait à la condition ($C_{1}$).\\
3) $H"\simeq \mathbb{Z}/d\mathbb{Z}\times \mathbb{Z}/n\mathbb{Z}$.
\end{proposition}
\begin{proof}
1) Posons $a:=t_{1}t_{2}$, $b:=t_{2}t_{3}$ et $c:=t_{3}t_{1}$. Il est clair que $abc=1$. De plus $acb=(t_{1}t_{2}t_{3})^{2}=1$, $cb=bc$. On voit de la même manière que $ac=ca$ et $ab=ba$: le groupe $H"$ est commutatif.\\
2) Comme $t_{1}at_{1}^{-1}=a^{-1}$ et $t_{1}ct_{1}^{-1}=c^{-1}$ et comme $G"=<H",t_{1}>$, nous voyons que $H"$ est un sous-groupe normal d'indice $2$ de $G"$ dont tous les éléments sont inversés par $t_{i}$ $(1\leqslant i \leqslant 3)$. Il en résulte que l'ordre de $ac$ est $ppcm(p,q)$ donc $r\,| \,ppcm(p,q)$. Par symétrie, on voit que $p\,| \,ppcm(q,r)$ et $q\,| \,ppcm(r,p)$: le triple d'entiers $(p,q,r)$ satisfait à la condition $(C_{1})$.\\
3) Soit $\pi$ un nombre premier et soit $\Pi$ la partie $\pi$-primaire de $|H"|$. Nous avons
\[
H"=\oplus_{\pi \in \mathcal{P}}\Pi
\]
où $\mathcal{P}$ désigne l'ensemble des nombres premiers positifs.\\
Posons $p=\pi^{n(p)}p'$, $q=\pi^{n(q)}q'$ et $r=\pi^{n(r)}r'$ où $p'$, $q'$ et $r'$ sont premiers à $p$. Posons $a':=a^{p'}$, $b':= b^{r'}$ et $c':=c^{q'}$. Alors $a'$ est d'ordre $\pi^{n(p)}$, $b'$ est d'ordre $\pi^{n(r)}$ et $c'$ est d'ordre $\pi^{n(q)}$.\\
Nous avons
\[
\Pi=<a',c'\,|\,a'^{\pi^{n(p)}}=c'^{\pi^{n(r)}}=[a',c']=(a'c')^{\pi^{n(q)}}=1>
\]
Nous pouvons supposer que $n(p)=n(r)$ et $n(q)\leqslant n(p)$ (voir l'appendice ). Nous effectuons des opérations élémentaires sur la matrice des relations de $\Pi$:
\[
\begin{pmatrix}
\pi^{n(p)} & 0\\
0 & \pi^{n(p)}\\
\pi^{n(q)} & \pi^{n(q)}
\end{pmatrix}
\to
\begin{pmatrix}
\pi^{n(p)} & -\pi^{n(p)}\\
0 & \pi^{n(p)}\\
\pi^{n(q)} & 0
\end{pmatrix}
\to
\begin{pmatrix}
\pi^{n(p)} & 0\\
0 & \pi^{n(p)}\\
\pi^{n(q)} & 0
\end{pmatrix}
\]
\[
\to
\begin{pmatrix}
\pi^{n(q)} & 0\\
0 & \pi^{n(p)}\\
\pi^{n(p)} & 0
\end{pmatrix}
\to
\begin{pmatrix}
\pi^{n(q)} & 0\\
0 & \pi^{n(p)}\\
0 & 0
\end{pmatrix}
\]
Il en résulte que $\Pi \simeq \mathbb{Z}/\pi^{n(q)}\mathbb{Z} \times \mathbb{Z}/\pi^{n(p)}\mathbb{Z}$. Comme $\pi^{n(q)}$ est la $\pi$-contribution  de $pgcd(p,q,r)$ et $\pi^{n(p)}$ est la $\pi$-contribution  de $ppcm(p,q,r)$, nous avons le résultat car $H"$ est le produit de ses sous-groupes de Sylow.
\end{proof}
Nous avons maintenant assez d'éléments pour le résultat suivant:
\begin{theorem}
Soient $W(p,q,r)$ un groupe de Coxeter et $R$ une représentation de réflexion  de $W(p,q,r)$. On suppose que le triple d'entiers $(p,q,r)$ satisfait à la condition (C). Soit $\alpha$ une racine de $v_{p}(X)$. On sait (voir l'appendice ) que l'on peut trouver $\beta$ racine de $v_{q}(X)$ et $\gamma$ racine de $v_{r}(X)$ de telle sorte que $\Delta=0$ et $\alpha l=\beta m$. Alors $G'$ est un groupe diédral d'ordre $2n$, où $n=ppcm(p,q,r)$ et $R$ est réductible.
\end{theorem}
\begin{proof}
D'après la proposition 1 de l'appendice (voir \cite{Z1}), on sait que le polynôme caractéristique de $t_{1}=s_{1}s_{2}s_{3}$ est $P_{t_{1}}=(X-1)^{2}(X+1)$, donc $t_{1}$ a comme valeurs propres sur $M'$: $-1$ et $+1$, d'où $(t'_{1})^{2}=id_{M'}$. D'après la proposition 4, $G'$ est isomorphe à un quotient du groupe $G"$ défini dans cette proposition. Il en résulte que $G'$ est une extension d'un groupe commutatif fini $H$ par un élément d'ordre $2$ qui inverse tous les éléments de $H$. De plus $G'$ opère fidèlement sur $M'$ par définition. Comme $<s_{i},s_{j}>$ $(1\leqslant i < j \leqslant 3)$ s'envoie fidèlement dans $G'$ et comme la représentation du groupe diédral $<s_{i},s_{j}>$ est absolument irréductible sur $M'$, nous voyons que $G'$ est un groupe diédral d'ordre $2n$ en appliquant la proposition 3 car toute représentation irréductible complexe de $D_{n}$ est réalisable sur $K_{0}$.
\end{proof}
\subsection{Le groupe $N(G)$}
Comme $G'$ est un groupe diédral d'ordre $2n$ et comme $N$ est un groupe commutatif libre tel que la représentation de $G'$ sur $N\otimes_{\mathbb{Z}}\mathbb{Q}$ est irréductible, nous allons étudier les représentations rationnelles irréductibles du groupe diédral $D_{n}$ d'ordre $2n$ $(n\geqslant 3)$. En particulier, nous allons voir qu'il n'y en a qu'une seule qui est fidèle.

Nous allons d'abord construire les représentations irréductibles rationnelles du groupe diédral $D_{n}$. D'après un résultat bien connu (voir \cite{S}), on a :\\
"Le nombre des classes de représentations irréductibles du groupe fini $G$ sur $\mathbb{Q}$ est égal au nombre des classes de conjugaison de sous-groupes cycliques de $G$."

Nous cherchons donc le nombre de classes de conjugaison de sous-groupes cycliques du groupe $G_{0}$
\[
G_{0}:=<g,s \,|\,g^{n}=s^{2}=1,sgs^{-1}=g^{-1}>\quad (n\geqslant 3).
\]
On pose $C:=<g>$. Il y a d'abord les classes de cardinal $1$ contenant chacune un sous-groupe de $C$: tous les sous-groupes de $C$ sont normaux dans $G_{0}$.\\
Si $n$ est impair, il y a une classe de conjugaison de sous-groupes d'ordre $2$ de $G_{0}$ et si $n$ est pair, il y a deux classes de conjugaison de sous-groupes d'ordre $2$ contenant les involutions non centrales de $G_{0}$.

Soit $\Phi_{n}(X)$ le n-ième polynôme cyclotomique . Le polynôme $\Phi_{n}(X)\in \mathbb{Z}[X]$, est irréductible et réciproque. On considère le corps $L=\mathbb{Q}[x]/(\Phi_{n}(X))$. C'est un espace vectoriel de dimension $\varphi(n)=\deg \Phi_{n}(X)$ sur le corps $\mathbb{Q}$. Le groupe de Galois $\mathcal{G}al(L/\mathbb{Q})$ est un groupe commutatif d'ordre $\varphi(n)$. Soit $\pi:\mathbb{Q}[x]\to L$ la projection canonique. On pose $g:=\pi(X)$ et l'on fait opérer $g$ sur $L$ par multiplication. Comme $n\geqslant 3$, $\varphi(n)$ est pair et le groupe de Galois contient un élément $s$ d'ordre $2$ qui opère comme $g\mapsto g^{-1}$. Le sous-corps $L_{0}$ des points fixes de $s$  est de degré $\frac{1}{2}\varphi(n)$. Il en résulte que $[L,s]$, sous-espace vectoriel de $L$ formé des éléments transformés en leurs opposés par $s$, est de dimension $\frac{1}{2}\varphi(n)$. Le sous-groupe de $GL(L)$ engendré par $g$ et $s$ est un groupe diédral d'ordre $2n$.

Nous construisons une base du $\mathbb{Q}$-espace vectoriel $L$ de la manière suivante: soit $e_{1}\in [L,s]-\{0\}$. Pour $1\leqslant i \leqslant\varphi(n)-1$, on pose $e_{i+1}=g.e_{i}$. Alors $(e_{1},e_{2},\cdots,e_{\varphi(n)})$ est une base de $L$ car $\Phi_{n}(X)$ est un polynôme irréductible.\\
Si $\Phi_{n}(X)=X^{\varphi(n)}+a_{1}X^{\varphi(n)-1}+\cdots+a_{1}X+1$, on a:
\[
g.e_{n}=-e_{1}-a_{1}e_{2}-\cdots-a_{1}e_{\varphi(n)}.
\]
De plus un calcul simple montre que $s(e_{i})=-g^{1-i}.e_{1}$ $(1\leqslant i \leqslant \varphi(n)$).\\
Il est clair que nous avons construit de cette manière une $\mathbb{Q}$-représentation irréductible et fidèle $R_{n}$ du groupe $G_{0}$. Soit maintenant $d\geqslant 3$ un diviseur de $n$ et soit $C_{d}:=<g^{n/d}>$. Alors $C_{d}\lhd G_{0}$ et $G_{0}/C_{d}$ est isomorphe à un groupe diédral d'ordre $2d$. Nous appliquons ce qui précède au groupe $G_{0}/C_{d}$ et nous obtenons ainsi une $\mathbb{Q}$-représentation irréductible du groupe $G_{0}$, de noyau $C_{d}$ et de degré $\varphi(d)$.\\
- Si $n$ est impair, alors $C_{1}$ est d'ordre $1$ et nous obtenons la représentation $R_{1}$ de degré $1$ où $s$ opère comme $-1$. Nous avons aussi dans ce cas la représentation $R_{0}$ où $s$ opère comme l'identité. D'après le résultat cité plus haut nous avons obtenu toutes les représentations rationnelles irréductibles de $G_{0}$.\\
- Si $n$ est pair, alors $C_{2}$ est d'ordre $2$ et nous obtenons la représentation $R_{2}$ de degré $\varphi(2)=1$ où $g$ et $s$ opèrent comme $-1$. Comme dans le cas impair, il y a la représentation $R_{1}$ de degré $1$ où $s$ opère comme $-1$ et $g$ comme l'identité; et enfin il y a la représentation $R_{1}\otimes R_{2}$ de degré $1$ où $s$ opère comme l'identité et $g$ opère comme $-1$.

Nous pouvons encore remarquer que toutes les représentations $R_{i}$ sont des $\mathbb{Z}$-représentations.

Toutes les représentations absolument simples de $G_{0}$ sont de degré $1$ ou $2$ et peuvent s'écrire sur le corps $K=\mathbb{Q}(\cos \frac{2\pi}{n})$. Celles qui sont fidèles sont celles pour lesquelles le caractère $\chi_{k}$ est tel que $\chi_{k}(g)=2\cos \frac{2k\pi}{n}$ avec $1\leqslant k <n$ et $pgcd(k,n)=1$. Toutes ces représentations sont conjuguées sous l'action du groupe de Galois $\mathcal{G}al(K/\mathbb{Q})$ et il y en a $\frac{1}{2}\varphi(n)$ une pour chaque valeur de $k$, mais si $k'$ est tel que $k+k'=n$, on a $\chi_{k}=\chi_{k'}$: les représentations obtenues sont équivalentes.
La représentation $R_{n}$ se décompose sur $K$ en somme de représentations irréductibles fidèles, toutes de degré $2$ et nous avons $\deg R_{n}=\varphi(n)=2(\frac{1}{2}\varphi(n))$. En particulier, nous voyons que l'indice de Schur vaut $1$ dans ce cas.
\begin{theorem}
Le groupe $N$ est un groupe commutatif libre de rang $\varphi(n)$.
\end{theorem}
\begin{proof}
C'est clair d'après ce qui précède. En effet d'abord $N$ est de type fini et sans torsion, donc c'est un groupe commutatif libre. Ensuite le groupe $G'$ (isomorphe au groupe diédral $D_{n}$ d'ordre $2n$) opère d'une manière fidèle et irréductible sur $N\otimes_{\mathbb{Z}}\mathbb{Q}$. D'après ce que nous savons sur la représentation $R_{n}$, nous avons le résultat.
\end{proof}
\subsection{Structure du groupe $G$}
Nous reprenons les notations précédentes. nous avons construit la base $\mathcal{E}=(e_{1},\cdots,e_{\varphi(n)})$ du $\mathbb{Q}$-espace vectoriel $L$, les éléments $g$ et $s$ opérant de la manière suivante: 
\[
g.e_{i}=e_{i+1}, (1\leqslant i <\, \varphi(n)),\,g.e_{\varphi(n)}=-e_{1}-a_{1}e_{2}-\cdots -a_{1}e_{\varphi(n)},
\]
\[
 s.e_{i}=-g^{1-i}.e_{1} (1\leqslant i <\, \varphi(n)).
\]
Nous pouvons considérer le réseau $\Lambda$ de $L$ engendré par $\mathcal{E}$. Les formules précédentes montrent que $G_{0}$ stabilise $\Lambda$. Nous formons maintenant le produit semi-direct: $G_{1}:=\Lambda \rtimes G_{0}$. Nous avons alors:
\begin{theorem}
Le groupe $G$ ($=G(\alpha,\beta,\gamma;l)$) est isomorphe à $G_{1}$.
\end{theorem}
\begin{proof}
Soit $f'$ un diviseur propre de $n$: $n=ff'$. Si $\zeta$ est une racine primitive n-ième de l'unité, on a $\zeta^{n}=1$ donc 
\[
\zeta^{ff'}-1=0=(\zeta^{f'})^{f}-1=(\zeta^{f'}-1)((\zeta^{f'})^{f-1}+\cdots+\zeta^{f'}+1).
\]
Comme $\zeta^{f'}\neq 1$, $\zeta^{f'}$ est racine du polynôme $X^{f-1}+\cdots+X+1$.
Soit $h$ un élément d'ordre $f$ de $C$. D'après la remarque précédente, $\forall e \in L$ on a $(h^{f-1}+h^{f-2}+\cdots+h+1)(e)=0$. Il en résulte que dans le produit semi-direct $\Lambda \rtimes G_{0}$, si $h$ est un élément d'ordre $f$ de $G_{1}$, alors $\forall e \in \Lambda$, l'élément $he$ de $\Lambda \rtimes G_{0}$ est d'ordre $f$.

On suppose que le triple d'entiers $(p,q,r)$ satisfait à la condition (C). On utilise les notations 1. De plus si $n$ est pair, on suppose (condition $(C_{1}))$ que 
$v_{2}(n)=v_{2}(p)=v_{2}(r)$ et $v_{2}(q) < v_{2}(n)$. D'après l'appendice, on peut trouver trois involutions non centrales dans $G_{0}$, $s_{1}$, $s_{2}$ et $s_{3}$ telles que $s_{1}s_{2}=g_{3}$ est d'ordre $r$, $s_{2}s_{3}=g_{1}$ est d'ordre $p$ et $s_{3}s_{1}=g_{2}$ est d'ordre $q$.

On va montrer que $G_{1}$ est engendré par trois involutions. Nous gardons les notations précédentes. Pour $s_{1}$, nous prenons le $s$ servant à définir $G_{0}$; pour $s_{2}$, nous prenons $s_{1}g^{u}$ avec $g^{u}$ d'ordre $p$ et pour $s_{3}$, nous prenons $s_{1}g^{v}e$ avec $g^{v}$ d'ordre $q$ et $e$ est le transformé par $g$ d'un vecteur de la base $\mathcal{E}$ de $L$.

Nous avons $s_{1}s_{2}=g^{u}$, donc $u=p_{1}k_{1}$ avec $pgcd(k_{1},p)=1$ car $g^{u}$ est d'ordre $p$. De même, $g^{v}$ est d'ordre $q$, $v=q_{1}k_{2}$ avec $pgcd(k_{2},q)=1$ et enfin $g^{v-u}$ est d'ordre $r$, $v-u=r_{1}k_{3}$ avec $pgcd(k_{3},r)=1$. L'existence de $k_{1}$, $k_{2}$ et $k_{3}$ est assurée puisque le triple d'entiers $(p,q,r)$ satisfait à la condition (C), donc aussi à la condition (D).

Appelons $G_{2}$ le sous-groupe de $G_{1}$ engendré par $s_{1}$, $s_{2}$ et $s_{3}$. Alors $G_{2}$ contient $s_{1}$, $g^{u}$ et $g^{v}e$. En particulier, comme $g^{u}$ et $g^{p_{1}}$ ont le même ordre $p$, $g^{p_{1}}$ est une puissance de $g^{u}$ et donc $g^{p_{1}}$ est dans $G_{2}$. Nous montrons que $e\in G_{2}$. Nous avons $g^{-p_{1}}g^{v}e=g^{v-p_{1}}e\in G_{2}$ et $g^{v}eg^{-p_{1}}=g^{v-p_{1}}g^{p_{1}}(e)\in G_{2}$. Il en résulte aussitôt que $g^{p_{1}}(e)-e\in G_{2}$. Une récurrence facile montre que $\forall k \in \mathbb{N}, g^{p_{1}k}(e)-e \in G_{2}$.\\
En effet, le résultat est vrai pour $k=1$. Ensuite si $k\geqslant 2$ et si le résultat est vrai pour $k-1$, alors $g^{p_{1}(k-1)}(e)-e \in G_{2}$, donc $g^{p_{1}}(g^{p_{1}(k-1)}(e)-e)g^{-p_{1}}=g^{p_{1}k}(e)-g^{p_{1}}(e) \in G_{2}$ d'où le résultat.\\
En particulier, comme $q=p_{1}r_{1}d$ et $r=p_{1}q_{1}d$, nous voyons que $g^{q}(e)-e \in G_{2}$ et $g^{r}(e)-e \in G_{2}$.\\
D'après la remarque initiale, nous savons que $g^{q}$ est racine du polynôme $X^{q_{1}-1}+X^{q_{1}-2}+\cdots+X+1$ et $g^{r}$ est racine du polynôme $X^{r_{1}-1}+X^{r_{1}-2}+\cdots+X+1$ dans l'anneau $\mathcal{L}(\Lambda)$. Il en résulte que $((g^{q})^{q_{1}-1}+\cdots+g^{q}+id_{L})(e)=0$ donc nous obtenons $q_{1}e \in G_{2}$ et de même $r_{1}e \in G_{2}$. Comme $pgcd(q_{1},r_{1})=1$, nous avons le résultat: $e \in G_{2}$. Nous en déduisons que $g^{u}$ et $g^{v}$ sont dans $G_{2}$, donc comme $<g>=<g^{u},g^{v}>$ nous voyons que $g \in G_{2}$.

Comme l'ensemble des transformés de $e$ par $<g>$ contient une base de $\Lambda$, nous avons le résultat $G_{2}\simeq G_{1}$. Il en résulte aussitôt que $G_{1}$ est isomorphe à un quotient du groupe $G$. Comme tous les quotients propres de $G$ sont finis, nous obtenons $G_{1}\simeq G$.
\end{proof}
\subsection{Appendice}
Dans cet appendice, nous démontrons un résultat d'arithmétique (voir \cite{St}) utilisé pour caractériser certains triples d'entiers qui donnent les groupes diédraux.
\begin{proposition}\label{propC1}
Les deux conditions suivantes $(C)$ et $(D)$ sur le triple d'entiers non nuls $(a_{1},a_{2},a_{3})$ sont équivalentes:
\begin{align*}
(C)
\begin{cases}
(C_{1}) & n=ppcm(a_{1},a_{2},a_{3})=ppcm(a_{i},a_{j}) (1\leqslant i\neq j \leqslant 3);\\
(C_{2}) & \parbox{11 cm}{%
$\exists i,j \in \mathbb{N}$ tels que $(1\leqslant i\neq j \leqslant 3)$ et $v_{2}(a_{i})=v_{2}(a_{j})=v_{2}(n)$;
si $|\{i,j,k\}|=3$, $v_{2}(a_{k})<v_{2}(n)$.}
\end{cases}\\
(D)
\begin{cases}
\exists c_{i}  \in \mathbb{Z} (1\leqslant i \leqslant 3) & \text{tels que}\\
(D_{1}) & \sum_{i=1}^{i=3}\frac{c_{i}}{a_{i}}\in \mathbb{Z};\\
(D_{2}) & pgcd(c_{i},a_{i})=1 (1\leqslant i \leqslant 3).
\end{cases}
\end{align*}
(Si $p$ est un nombre premier et $a$ un entier, $v_{p}(a)$ désigne l'exposant de la plus grande puissance de $p$ qui divise $a$.)
\end{proposition}
\begin{proof}
1) Nous montrons d'abord que $(D)\Longrightarrow (C)$.\\
Nous pouvons écrire $\frac{c_{1}}{a_{1}}+\frac{c_{2}}{a_{2}}=m-\frac{c_{3}}{a_{3}}$ où $m$ est un entier. Comme $\frac{c_{3}}{a_{3}}$ est une fraction réduite, $a_{3}| ppcm(a_{1},a_{2})$, donc $ppcm(a_{1},a_{2},a_{3})=ppcm(a_{1},a_{2})=ppcm(a_{i},a_{j})$\\ $(i\neq j)$ par symétrie et la condition $(C_{1})$ est satisfaite.

Nous pouvons écrire $a_{1}=\epsilon_{1}p_{1}^{\alpha_{1}}\ldots p_{r}^{\alpha_{r}}$, $a_{2}=\epsilon_{2}p_{1}^{\beta_{1}}\ldots p_{r}^{\beta_{r}}$ et $a_{3}=\epsilon_{3}p_{1}^{\gamma_{1}}\ldots p_{r}^{\gamma_{r}}$ où chaque $p_{i}$ est un nombre premier, $\alpha_{i},\beta_{i},\gamma_{i}$ sont des entiers $\geqslant 0$ et $\epsilon_{i}\in \{-1,+1\}$ $(1\leqslant i \leqslant 3 $). nous avons $ppcm(a_{1},a_{2})=p_{1}^{sup(\alpha_{1},\beta_{1})}\ldots p_{r}^{sup(\alpha_{r},\beta_{r})}$ et de même pour $ppcm(a_{1},a_{3})$ et $ppcm(a_{2},a_{3})$. Il en résulte que pour $1 \leqslant i \leqslant 3$ nous avons $sup(\alpha_{i},\beta_{i})$ = $sup(\alpha_{i},\gamma_{i})$ = $sup(\beta_{i},\gamma_{i})$, donc nous avons nécessairement: deux des entiers $\alpha_{i},\beta_{i},\gamma_{i}$ sont égaux et le troisième est plus petit. En particulier, il existe $i$ et $j$ $(1 \leqslant i \neq j \leqslant 3)$ tels que $v_{2}(a_{i})=v_{2}(a_{j})=v_{2}(n)$ et si $|\{i,j,k\}|=3$, $v_{2}(a_{k})\leqslant v_{2}(n)$. Supposons que $v_{2}(n)\geqslant 1$ et $v_{2}(a_{1})=v_{2}(a_{2})=v_{2}(a_{3})=v_{2}(n)$. D'après la condition $(D_{2})$ chaque $c_{i}$ est impair et si nous définissons $b_{i}$ par $n=a_{i}b_{i}$ $(1 \leqslant i \leqslant 3)$, alors chaque $b_{i}$ est impair d'où $\frac{c_{1}}{a_{1}}+\frac{c_{2}}{a_{2}}+\frac{c_{3}}{a_{3}}=\frac{1}{n}\sum_{i=1}^{i=3}(b_{i}c_{i})$. Dans ces conditions $n$ est pair et $\sum_{i=1}^{i=3}(b_{i}c_{i})$ est impair, donc $\frac{c_{1}}{a_{1}}+\frac{c_{2}}{a_{2}}+\frac{c_{3}}{a_{3}}$ ne peut pas être entier. la condition $(C_{2})$ est satisfaite.

2) Nous montrons maintenant que $(D)\Longrightarrow (C)$. la démonstration de trois lemmes sera faite ultérieurement. Nous procédons par étapes.\\
\textbf{Etape 1.}\\
\begin{notation}
Comme ci-dessus nous définissons $b_{i}$ par $n=a_{i}b_{i}$ $(1\leqslant i \leqslant 3)$ et l'hypothèse faite implique que $pgcd(b_{i},b_{j})=1$ si $1 \leqslant i \neq j \leqslant 3$.
\end{notation}

Posons $w := pgcd(a_{1},a_{2},a_{3})$. Alors nous avons:
\[
n=b_{1}b_{2}b_{3},a_{1}=b_{2}b_{3}w,a_{2}=b_{3}b_{1}w,a_{3}=b_{1}b_{2}w.
\]

La condition $(D_{1})$ peut s'écrire $\sum_{i=1}^{i=3}\frac{b_{i}c_{i}}{n} \in \mathbb{Z}$ et $(D_{1})$ sera satisfaite si nous supposons que:
\[
\sum_{i=1}^{i=3}b_{i}c_{i}=n.
\]

Comme $pgcd(b_{i},b_{j})=1$ si $i\neq j$, l'idéal de $\mathbb{Z}$ engendré par $b_{1},b_{2},b_{3}$ est $\mathbb{Z}$ tout entier et il existe des entiers $c_{1},c_{2},c_{3}$ tels que:
\[
b_{1}c_{1}+b_{2}c_{2}+b_{3}c_{3}=n.
\]
 Nous choisissons $c_{1}$ de telle sorte que $1\leqslant c_{1} < a_{1}$ et $pgcd(c_{1},a_{1})=1$. L'équation précédente devient:
 \[
 b_{2}c_{2}+b_{3}c_{3}=n-b_{1}c_{1}=b_{1}(b_{2}b_{3}w-c_{1}).
\]
 On est ainsi amené à chercher les solutions de l'équation $(S)$:
 \[
 b_{2}x+b_{3}y=b_{1}(b_{2}b_{3}w-c_{1})
 \]
 \[
  pgcd(x,a_{2})=pgcd(y,a_{3})=1.
 \]
 
 Nous avons $pgcd(b_{2},c_{1})=pgcd(b_{3},c_{1})=1$. Il en résulte aussitôt que, puisque $pgcd(b_{i},b_{j})=1$ si $i\neq j$: $pgcd(x,b_{3})=pgcd(y,b_{2})=1$.\\
 Pour la même raison que ci-dessus nous avons \emph{l'égalité fondamentale}:
 \[
 pgcd (x,b_{1})=pgcd (y,b_{1})=1.
\]
Soit $(x_{0},y_{0})$ une solution particulière de $(S)$. Alors toutes les solutions de $(S)$ sont données par:
\[
x=x_{0}+\rho b_{3},y=y_{0}-\rho b_{2}\quad (\rho \in \mathbb{Z}).
\]

Comme $a_{2}=b_{1}b_{3}w$ (resp. $a_{3}=b_{1}b_{2}w$ ), $ pgcd(x,a_{2})=1$ équivaut à\\ $pgcd(x,b_{1}w)=1$ (resp. $pgcd(y,a_{3})=1$ équivaut à $pgcd(y,b_{1}w)=1$).

\textbf{Etape 2.}\emph{On change la nature du problème.}\\
\begin{notation}
Si $m$ est un entier, on pose :
\[
P(m):=\{p | p \; \text{nombre premier} ,p|m\}
\]
\end{notation}
Nous avons alors: $pgcd(x,b_{1}w)=1$ si et seulement si $\forall p \in P(b_{1}w)$, $pgcd(x,p)=1$.
Ceci nous amène à: Chercher les valeurs de $\rho$ dans $\mathbb{Z}$ pour lesquelles
\[
(F)\begin{cases}
\forall p\in P(b_{1}w)-P(b_{3}), & p\nmid x=x_{0}+\rho b_{3};\\
\forall q\in P(b_{1}w)-P(b_{2}), & q\nmid y=y_{0}-\rho b_{2}
\end{cases}
\]
la condition $(F)$ est satisfaite.\\
Posons $R:=\{ \rho |\: \rho \in \mathbb{Z},\: (F) \: \text{est satisfaite}\}$. Nous voulons montrer que $R\neq \emptyset$, car si $\rho \in R$ nous aurons $pgcd(x,a_{2})=1$ et $pgcd(y,a_{3})=1$. Pour cela, nous considérons l'ensemble $S:=\mathbb{Z}-R$. Posons:
\[
S_{x}:=\{\rho |\: \rho \in \mathbb{Z},\: \exists p\in P(b_{1}w)-P(b_{3}),p|x_{0}+\rho b_{3}\}
\]
\[
S_{y}:=\{\rho |\: \rho \in \mathbb{Z},\: \exists q\in P(b_{1}w)-P(b_{2}),q|y_{0}-\rho b_{2}\}.
\]
Nous posons aussi:
\[
P_{1}:=P(b_{1})\sqcup (P(b_{2})\cap P(w))\sqcup (P(b_{3})\cap P(w))
\]
d'où $P(b_{1}w)=P_{1}\cup P_{0}$ avec $P_{0}:=\{s_{1},\ldots,s_{\delta}\}$ (remarquons que $\delta$ peut être $0$, auquel cas on a $P_{0}=\emptyset$). Si $p \in P(b_{1}w)-P(b_{3})$, nous posons
\[
S_{x}(p):=\{\rho |\rho \in \mathbb{Z},p|x_{0}+\rho b_{3}\}
\]
et si $q \in P(b_{1}w)-P(b_{2})$, nous posons
\[
S_{y}(q):=\{\rho |\rho \in \mathbb{Z},p|y_{0}-\rho b_{2}\}
\]
de telle sorte que $S_{x}=\cup_{p \in P(b_{1}w)-P(b_{3})}S_{x}(p)$ et $S_{y}=\cup_{q \in P(b_{1}w)-P(b_{2})}S_{y}(q)$. Enfin nous avons $S=S_{x}\cup S_{y}$.

\textbf{Etape 3.}\emph{Les ensembles $S_{x}(p)$ et $S_{y}(q)$ sont des classes à gauche de $\mathbb{Z}   \pmod {b_{1}w}$.}

D'après le lemme \ref{lemmaC4}  nous avons : $\forall p \in P(b_{1}w)-P(b_{3})$, $S_{x}(p)=\{\rho_{0}(x,p)+\varphi p| \varphi \in \mathbb{Z}\}$ où $0\leqslant \rho_{0}(x,p) < p$, ce que nous pouvons écrire:
\[
S_{x}(p)=\{\rho_{0}(x,p)+p\varphi_{x}(p)+b_{1}w\varphi' | 0\leqslant \varphi_{x}(p)<b_{1}wp^{-1},\varphi' \in \mathbb{Z}\}.
\]
De même $\forall q \in P(b_{1}w)-P(b_{2})$, $S_{x}(q)=\{\rho_{0}(y,q)+ \theta q| \theta \in \mathbb{Z}\}$ où $0\leqslant \rho_{0}(y,q) < p$, ce que nous pouvons écrire:
\[
S_{y}(q)=\{\rho_{0}(y,q)+q\theta_{y}(q)+b_{1}w\theta' | 0\leqslant \theta_{y}(q)<b_{1}wq^{-1},\theta'\in \mathbb{Z}\}.
\]
Il en résulte que $S_{x}(p)$ est la réunion de $b_{1}wp^{-1}$ classes à gauche de $\mathbb{Z}\pmod {b_{1}w}$ et $S_{y}(q)$ est la réunion de $b_{1}wq^{-1}$ classes à gauche de $\mathbb{Z}\pmod {b_{1}w}$.

Remarquons que si $p\in P(b_{1})$ alors $S_{x}(p)=S_{y}(p)$; de plus, si $s_{i}\in P_{0}$, $S_{x}(s_{i})$ et $S_{y}(s_{i})$ sont deux classes à gauche distinctes de $\mathbb{Z}\pmod {b_{1}w}$: nous avons $S_{x}(s_{i})\cap S_{y}(s_{i})=\emptyset$.

\textbf{Etape 4.} On pose si $p\in P(b_{1}w)-P(b_{3})$ et si $q\in P(b_{1}w)-P(b_{2})$:
\[
T_{x}(p):=\{\rho_{0}(x,p)+p\varphi_{x}(p)|0\leqslant\varphi_{x}(p)<b_{1}wp^{-1}\}
\]
\[
T_{y}(q):=\{\rho_{0}(y,q)+q\theta_{y}(q)|0\leqslant\theta_{y}(q)<b_{1}wq^{-1}\}
\]
Enfin on pose\\ $T_{x}:=\bigcup_{p\in P(b_{1}w)-P(b_{3})}T_{x}(p)$, $T_{y}:=\bigcup_{q\in P(b_{1}w)-P(b_{2})}T_{y}(q)$ et $T:=T_{x}\bigcup T_{y}$.\\
Pour montrer que $S\neq\mathbb{Z}$, il suffit de montrer que $|T|<b_{1}w$.

\textbf{Etape 5.}\emph{Calcul de $|T|$}.\\
\emph{On a}:
\[
|T|=b_{1}w(1-\prod_{p\in P_{1}}(\frac{p-1}{p})\prod_{s\in P_{0}}(\frac{s-2}{s})).
\]
\emph{En particulier} $|T|<b_{1}w$.
\begin{proof}
La condition $C_{2}$ est équivalente au fait que $2\notin P_{0}$, donc la formule à démontrer implique que 
$|T|<b_{1}w$.

Nous posons:
\[
V:=\bigcup_{p\in P_{1}}(T_{x}(p)\bigcup T_{y}(p)).
\]
Si $p\in P(b_{1})$, alors $T_{x}(p)=T_{y}(p)$; si $p\in P(b_{1}w)-(P(b_{1})\bigcup P(b_{2}))$ alors $T_{x}(p)=\emptyset$ et il n'y a que $T_{y}(p)$ à considérer; si $p\in P(b_{1}w)-(P(b_{1})\bigcup P(b_{3}))$ alors $T_{y}(p)=\emptyset$ et il n'y a que $T_{x}(p)$ à considérer. En utilisant les lemmes \ref{lemmaC5} et \ref{lemmaC6} nous obtenons, en ordonnant $P_{1}$ par ordre croissant:
\[
|V|=b_{1}w(\sum_{p\in P_{1}}\frac{1}{p}-\sum_{p_{1}<p_{2}}\frac{1}{p_{1}p_{2}}+\sum_{p_{1}<p_{2}<p_{3}}\frac{1}{p_{1}p_{2}p_{3}}-\ldots)
\]
ce qui peut s'écrire aussi:
\[
|V|=b_{1}w(1-\prod_{p\in P_{1}}(\frac{p-1}{p})).
\]
Nous posons:
\[
A:=\prod_{p\in P_{1}}(\frac{p-1}{p}).
\]
Nous nous intéressons maintenant aux ensembles $T_{x}(s_{i})$ et $T_{y}(s_{i})$ $(1\leqslant i \leqslant\delta)$.\\
Nous avons d'abord $T_{x}(s_{i})\bigcap  T_{y}(s_{i})=\emptyset$ $(1\leqslant i \leqslant\delta)$ et ensuite $|T_{x}(s_{i})|=|T_{y}(s_{i})|=\frac{b_{1}w}{s_{i}}$. Soit $s\in P_{0}$. Nous allons considérer $T_{x}(s)$ puis $T_{x}(s)\cap V$ (resp. $T_{y}(s)$ puis $T_{y}(s)\cap V$). En utilisant les lemmes \ref{lemmaC4} et \ref{lemmaC5} nous trouvons $|T_{x}(s)|=\frac{b_{1}}{s}$ et si $p\in P_{1}$, $|T_{x}(s)\cap T_{x}(p)|=\frac{b_{1}w}{sp}$. Finalement, en tenant compte des diverses intersections, nous obtenons:
\[
|T_{x}(s)\cap V|=\frac{b_{1}w}{s}(\sum_{p\in P_{1}}\frac{1}{p}-\sum_{p_{1}<p_{2}}\frac{1}{p_{1}p_{2}}+\sum_{p_{1}<p_{2}<p_{3}}\frac{1}{p_{1}p_{2}p_{3}}-\ldots)=\frac{b_{1}w}{s}(1-A).
\]
Nous avons ainsi $\frac{2b_{1}w}{s}(1-A)$ correspondant à $T_{x}(s)$ et $T_{y}(s)$. Lorsqu'il n'y a qu'un seul $s \in P_{0}$ nous avons la contribution:
\[
\frac{2b_{1}w}{s}(1-\sum_{p\in P_{1}}\frac{1}{p}+\sum_{p_{1}<p_{2}}\frac{1}{p_{1}p_{2}}-\ldots)=\frac{2b_{1}w}{s}A
\]

Soient maintenant $s$ et $s'$ deux éléments distincts de $P_{0}$. Nous avons alors quatre intersections qui ont le même cardinal:
\[
|T_{x}(s)\cap T_{x}(s')|=\frac{b_{1}w}{ss'}=|T_{x}(s)\cap T_{y}(s')|=|T_{y}(s)\cap T_{x}(s')|=|T_{y}(s)\cap T_{y}(s')|
\] 
donc nous devons enlever $4\frac{b_{1}w}{ss'}A$.\\
Un raisonnement semblable montre que si $s_{i_{1}},s_{i_{2}},\ldots,s_{i_{m}}$ sont $m$ éléments distincts de $P_{0}$, alors nous avons $2^m$ intersections qui ont le même cardinal: il faut préciser le sous-ensemble $I$ de $\{i_{1},i_{2},\ldots,i_{m}\}$ pour lequel on considère $T_{x}(s_{i})$ $i\in I$ et pour $j\in \{i_{1},i_{2},\ldots,i_{m}\}-I$, pour lequel on considère $T_{y}(s_{j})$. De la même manière que ci-dessus, nous obtenons $\frac{2^{m}b_{1}wA}{s_{i_{1}}s_{i_{2}}\ldots s_{i_{m}}}$.\\
finalement nous obtenons pour les ensembles qui contiennent un élément de $P_{0}$:
\begin{align*}
&& 2b_{1}w(\sum_{s\in P_{0}}\frac{1}{s})A-4b_{1}w(\sum_{i_{1}<i_{2}}\frac{1}{s_{i_{1}}s_{i_{2}}})A+8b_{1}w(\sum_{i_{1}<i_{2}<i_{3}}\frac{1}{s_{i_{1}}s_{i_{2}}s_{i_{3}}})A-\ldots\\
& = & b_{1}wA(2(\sum_{s\in P_{0}}\frac{1}{s})-4(\sum_{i_{1}<i_{2}}\frac{1}{s_{i_{1}}s_{i_{2}}})+8(\sum_{i_{1}<i_{2}<i_{3}}\frac{1}{s_{i_{1}}s_{i_{2}}s_{i_{3}}})\ldots)
\end{align*}
En ajoutant et en retranchant $1$ dans la grande parenthèse, nous trouvons:
\[
b_{1}wA(1-\prod_{s\in P_{0}}(\frac{s-2}{s})).
\]
Il en résulte que:
\[
|T|=|V|+b_{1}wA((1-\prod_{s\in P_{0}}(\frac{s-2}{s}))
\]
ce qui peut s'écrire:
\[
|T|=b_{1}w(1-\prod_{p\in P_{1}}(\frac{p-1}{p})\prod_{s\in P_{0}}(\frac{s-2}{s})).
\]
\end{proof}
Ceci termine la preuve de la proposition.
\end{proof}
Nous avons eu besoin dans la démonstration du résultat précédent des trois lemmes élémentaires suivants que l'on peut trouver dans \cite{St}:
\begin{lemma}\label{lemmaC4}
Soient $a$ et $b$ deux entiers premiers entre eux et soit $c$ un autre entier. Alors l'ensemble des entiers $\rho$ tels que $b|c+a\rho$ est une classe à gauche de $\mathbb{Z} \pmod{b}$.
\end{lemma}
\begin{proof}
Soit $\rho$ dans $\mathbb{Z}$ tel que $b|c+a\rho$: il existe $m$ dans $\mathbb{Z}$ tel que $bm=c+a\rho$. On cherche donc les couples $(m,\rho)\in \mathbb{Z}\times\mathbb{Z}$ tels que 
\[
bm-a\rho=c
\]
soit satisfaite. Comme $pgcd(a,b)=1$, l'équation précédente a toujours des solutions. Si $(m_{0},\rho_{0})$ est une solution particulière, alors toutes les solutions sont données par: $m=m_{0}+\theta a$, $\rho=\rho_{0}+\theta b$ avec $\theta \in \mathbb{Z}$, d'où le résultat.
\end{proof}
\begin{remark}
Nous pouvons choisir $\rho_{0}$ de telle sorte que $0\leqslant \rho_{0}<b$ si $b$ est positif. Nous avons toujours fait un tel choix.
\end{remark}
\begin{lemma}\label{lemmaC5}
Soient $a$ et $b$ deux entiers premiers entre eux et $a_{0}$ et $b_{0}$ deux autres entiers. on pose: $A:=\{a_{0}+\rho a | \rho \in \mathbb{Z}\}$ et $B:=\{b_{0}+\rho b |  \rho \in \mathbb{Z}\}$. Alors il existe un entier $\rho_{0}$ tel que $A\cap B=\{a_{0}+\rho_{0}a+\pi ab | \pi \in \mathbb{Z}\}$.
\end{lemma}
\begin{proof}
Soit $x \in A\cap B$: $x=a_{0}+\rho a=b_{0}+\theta b$. Nous obtenons 
\[\rho a-\theta b=b_{0}-a_{0}.\]
 Comme $pgcd(a, b)=1$, l'équation précédente a toujours des solutions. En particulier $ A\cap B$ est non vide. Soit $(\rho_{0},\theta_{0})$ une solution particulière, alors toutes les solutions sont données par $\rho=\rho_{0}+\pi b$, $\theta=\theta_{0}+\pi a$ avec $\pi \in \mathbb{Z}$. Il en résulte que 
 \[
 A\cap B=\{a_{0}+\rho_{0}a+\pi ab | \pi \in \mathbb{Z}\}.
 \]
\end{proof}
\begin{lemma}\label{lemmaC6}
Soient $(E_{i})_{1\leqslant i \leqslant n}$ une famille d'ensembles finis et $E=\cup_{1\leqslant i \leqslant n}E_{i}$. Alors on a:
\[
|E|=\sum_{1}^{n}|E_{i}|-\sum_{1\leqslant i_{1}<i_{2}<n}|E_{i_{1}}\cap E_{i_{2}}|+\ldots+(-1)^{k-1}\sum_{1\leqslant i_{1}<i_{2}<\ldots<i_{k}\leqslant n}|E_{i_{1}}\cap E_{i_{2}}\cap\ldots\cap E_{i_{k}}|+\ldots
\]
\end{lemma}
\begin{proof}
Elle se fait par récurrence sur $n$.\\ Si $n=2$, il est clair que $|E|=|E_{1}|+|E_{2}|-|E_{1}\cap E_{2}|$.\\
Supposons $n \geqslant 3$ et le résultat vrai pour $n-1$. Nous pouvons écrire 
\[
E=(\cup_{1\leqslant i \leqslant n-1}E_{i})\cup E_{n}
\]
donc, d'après le cas $n=2$, nous trouvons:
\[
|E|=|\cup_{1\leqslant i \leqslant n-1}E_{i}|+|E_{n}|-|(\cup_{1\leqslant i \leqslant n-1}E_{i})\cap E_{n}|
\]
Nous avons $(\cup_{1\leqslant i \leqslant n-1}E_{i})\cap E_{n}=\cup_{1\leqslant i \leqslant n-1}(E_{i}\cap E_{n})$ d'où d'après l'hypothèse de récurrence:
\[
|\cup_{1\leqslant i \leqslant n-1}(E_{i}\cap E_{n})|=\sum_{i=1}^{n-1}|E_{i}\cap E_{n}|-\sum_{1\leqslant i_{1}<i_{2}\leqslant n-1}|E_{i_{1}}\cap E_{i_{2}}\cap E_{n}|+\ldots
\]
En regroupant les termes, nous avons la formule annoncée.
\end{proof}

Nous gardons les notations de la proposition et nous supposons ici que $a_{1}$, $a_{2}$ et $a_{3}$ sont positifs. On peut choisir $x_{0}$ de telle sorte que $0<x_{0}<b_{3}$. On a choisi $\rho$ de telle sorte que $0\leqslant \rho < b_{1}w$. dans ces conditions on a:
\[
0<x=x_{0}+\rho b_{3}\leqslant b_{3}-1+(b_{1}w-1)b_{3}=b_{1}b_{3}w-1=a_{2}-1
\]
donc $0<x<a_{2}$.\\
Nous avons $\frac{c_{1}}{a_{1}}+\frac{x}{a_{2}}+\frac{y}{a_{3}}=1$. Il en résulte que $\frac{y}{a_{3}}=1-\frac{c_{1}}{a_{1}}-\frac{x}{a_{2}}$ donc, comme $\frac{c_{1}}{a_{1}}>0$ et $\frac{x}{a_{2}}>0$, on a $\frac{y}{a_{3}}<1$. Ensuite $\frac{c_{1}}{a_{1}}+\frac{x}{a_{2}}<2$ donc $-1<\frac{y}{a_{3}}$. Nous obtenons ainsi $-a_{3}<y<a_{3}$. Nous appelons \emph{réduite} toute solution $(c_{1},c_{2},c_{3})$ de $\sum_{i=1}^{3}b_{i}c_{i}=n$ telle que:
\[
0<c_{1}<a_{1},\,0<c_{2}<a_{2},\,|c_{3}|<a_{3}.
\]
Nous avons alors la
\begin{proposition}
Le nombre de solutions réduites de $\sum_{i=1}^{3}b_{i}c_{i}=n$ est:
\[
\varphi(n)w\prod_{s\in P_{0}}(\frac{s-2}{s}).
\]
\end{proposition}
\begin{proof}
Nous avons $\varphi(a_{1})$ possibilités pour $c_{1}$, donc le nombre de solutions réduites de $\sum_{i=1}^{3}b_{i}c_{i}=n$ est:
\[
\varphi(a_{1})b_{1}w\prod_{p\in P_{1}}(\frac{p-1}{p})\prod_{s\in P_{0}}(\frac{s-2}{s}).
\]
Nous posons:
\[
Q:=\bigsqcup_{1\leqslant i \leqslant3}(P(b_{i})\cap P(w))
\]
\[
Q_{i}:=P(b_{i})-P(w) \; (1\leqslant i \leqslant 3).
\]
Alors
\[
P(a_{1})=Q\bigsqcup Q_{2}\bigsqcup Q_{3}\bigsqcup P_{0}
\]
et des formules semblables pour $P(a_{2})$ et $P(a_{3})$. De plus $P_{1}=Q\bigsqcup Q_{1}$, d'où:
\[
\prod_{p\in P_{1}}(\frac{p-1}{p})=\prod_{q\in Q}(\frac{q-1}{q})\prod_{p_{1}\in Q_{1}}(\frac{p_{1}-1}{p_{1}}).
\]
Enfin
\[
\varphi(a_{1})=b_{2}b_{3}w\prod_{q\in Q}(\frac{q-1}{q})\prod_{p_{2}\in Q_{2}}(\frac{p_{2}-1}{p_{2}})\prod_{p_{3}\in Q_{3}}(\frac{p_{3}-1}{p_{3}})\prod_{s\in P_{0}}(\frac{s-1}{s})
\]
donc le nombre de solutions de $\sum_{i=1}^{3}b_{i}c_{i}=n$ est:
\[
nw\prod_{q\in Q}(\frac{q-1}{q})\prod_{p_{1}\in Q_{1}}(\frac{p_{1}-1}{p_{1}})\prod_{p_{2}\in Q_{2}}(\frac{p_{2}-1}{p_{2}})\prod_{p_{3}\in Q_{3}}(\frac{p_{3}-1}{p_{3}})\prod_{s\in P_{0}}(\frac{s-1}{s})\prod_{s\in P_{0}}(\frac{s-2}{s}).
\]
Comme $n=b_{1}b_{2}b_{3}w$, nous avons $P(n)=Q\bigsqcup Q_{1}\bigsqcup Q_{2}\bigsqcup Q_{3}\bigsqcup P_{0}$ d'où:
\[
\varphi(n)=n\prod_{q\in Q}(\frac{q-1}{q})\prod_{p_{1}\in Q_{1}}(\frac{p_{1}-1}{p_{1}})\prod_{p_{2}\in Q_{2}}(\frac{p_{2}-1}{p_{2}})\prod_{p_{3}\in Q_{3}}(\frac{p_{3}-1}{p_{3}})\prod_{s\in P_{0}}(\frac{s-1}{s})
\]
Finalement le nombre de solutions de $\sum_{i=1}^{3}b_{i}c_{i}=n$ est:
\[
\varphi(n)w\prod_{s\in P_{0}}(\frac{s-2}{s}).
\]
\end{proof}
Nous appliquons la proposition \ref{propC1} à un résultat sur la génération des groupes diédraux.
\begin{proposition}\label{propC3}
Soit $D:=<g,s|g^n=s^2=1,sgs^{-1}=g^{-1}>$ un groupe diédral d'ordre $2n$ avec $n\geqslant 3$. On appelle $G$ son sous-groupe cyclique d'ordre $n$. Soient $a_{1},a_{2},a_{3}$ trois entiers $>0$. Les deux conditions suivantes sont équivalentes:
\begin{itemize}
  \item (A) Il existe trois involutions non centrales $s_{i}$ $(1\leqslant i \leqslant 3)$ de $D$ telles que si $s_{i}s_{j}=g_{k}$ $(1 \leqslant i \neq j \leqslant 3, \; |\{i,j,k\}|=3)$ avec $g_{k}$ d'ordre $a_{k}$, on ait $G=<g_{i},g_{j}> \;(1 \leqslant i \neq j \leqslant 3)$.
  \item (B) Le triple d'entiers $(a_{1},a_{2},a_{3})$ satisfait à la condition (C).
\end{itemize}
\end{proposition}
\begin{proof}
1) Montrons que $(A)\Longrightarrow(B)$. Comme $G=<g_{i},g_{j}>$ $(1 \leqslant i \neq j \leqslant 3)$, nécessairement $ppcm(a_{i},a_{j})=n$ $(1 \leqslant i \neq j \leqslant 3)$ donc la condition $(C_{1})$ est satisfaite. Supposons maintenant $n$ pair, $n=2m$. Nous avons vu qu'il existe $i$ et $j$ $(1\leqslant i \leqslant 3)$ tels que $v_{2}(a_{i})=v_{2}(a_{j})=v_{2}(n)$. Les conjugués de $s_{3}$ dans $D$ sont les $g^{2l}s_{3}$ $(0\leqslant l \leqslant m-1)$ et l'autre classe d'involutions non centrales est l'ensemble des $g^{2l+1}s_{3}$ $(0\leqslant l \leqslant m-1)$. Nous avons $s_{1}s_{2}=g_{3}$, $s_{1}s_{3}=g_{2}$, $s_{2}s_{3}=g_{1}$ donc $g_{3}=g_{2}g_{1}^{-1}$. Si nous supposons que $v_{2}(a_{1})=v_{2}(a_{2})=v_{2}(n)$, alors $s_{1}=g^{2\alpha+1}s_{3}$ et $s_{2}=g^{2\beta+1}s_{3}$. Il en résulte que $g_{3}=s_{1}s_{2}=g^{2\alpha-2\beta}$ et alors $v_{2}(a_{3})<v_{2}(n)$: la condition $(C_{2})$ est satisfaite.\\
2) Montrons que $(B)\Longrightarrow(A)$. Nous avons vu que les conditions $(C)$ et $(D)$ étaient équivalentes. Soient $c_{i}\in \mathbb{Z}$ $(1 \leqslant i \leqslant 3)$ tels que $pgcd(c_{i},a_{i})=1$ $(1 \leqslant i \leqslant 3)$ et $\sum_{1 \leqslant i \leqslant 3}\frac{c_{i}}{a_{i}}\in\mathbb{Z}$. Nous reprenons les notations de la proposition \ref{propC1}: $n=a_{i}b_{i}$ et si $pgcd(a_{1},a_{2},a_{3})w$, $n=b_{1}b_{2}b_{3}w$, $a_{1}=b_{2}b_{3}w$, $a_{2}=b_{3}b_{1}w$ et $a_{3}=b_{1}b_{2}w$. Les éléments d' ordre $a_{i}$ de $G$ sont les $g^{b_{i}x_{i}}$ avec $pgcd(x_{i},a_{i})=1$ $(1 \leqslant i \leqslant 3)$.\\
On cherche trois involutions non centrales $s_{1}, s_{2},s_{3}$ de $D$ telles que $s_{i}s_{j}=g_{k}$ soit d'ordre $a_{k}$ $(1 \leqslant i < j \leqslant 3,|\{i,j,k\}|=3)$. On cherche donc $x_{1},x_{2},x_{3}$ tels que $pgcd(x_{i},a_{i})=1$ $(1 \leqslant i  \leqslant 3)$ et comme $g_{1}=g_{3}^{-1}g_{2}$ $g^{b_{1}x_{1}}=g^{b_{2}x_{2}-b_{3}x_{3}}$ ou encore $b_{1}x_{1}\equiv b_{2}x_{2}-b_{3}x_{3} \pmod{n}$. D'après la proposition \ref{propC1}, il existe des entiers $c_{1},c_{2},c_{3}$ tels que $pgcd(c_{i},a_{i})=1$ $(1 \leqslant i  \leqslant 3)$  et $b_{1}c_{1}-b_{2}c_{2}+b-{3}c_{3}=n$. Soit $s_{1}$ une involution non centrale quelconque de $D$. Posons $g_{2}:=g^{b_{2}c_{2}}$, $g_{3}:=g^{b_{3}c_{3}}$, $s_{2}:=s_{1}g_{2}$ et $s_{3}:=s_{1}g_{3}$. La condition (A) est satisfaite.
\end{proof}

\begin{remark}
Jean-Yves Hée a généralisé la proposition \ref{propC1} en remplaçant $\mathbb{Z}$ par un anneau factoriel quelconque  et $3$ par un nombre quelconque de variables ($\geqslant 3$).
\end{remark}

\begin{center}
Université de Picardie Jules Verne\\
 Pôle Scientifique\\
Laboratoire LAMFA, UMR CNRS 7352\\
33, rue Saint Leu\\
80039 Amiens Cedex\\
francois.zara@u-picardie.fr
\end{center}

\end{document}